\documentclass[12pt]{article}
\usepackage[utf8]{inputenc}

\usepackage[affil-it]{authblk}
\usepackage{amsthm,amsmath}
\usepackage{paralist,tikz,mathtools}

\newtheorem{theorem}{Theorem}
\newtheorem{lemma}{Lemma}
\newtheorem{definition}{Definition}

\providecommand{\keywords}[1]
{\small	
  \textbf{\textit{Keywords---}} #1}

\newcommand\pint[2]{{\big\langle #1, #2 \big\rangle}}

\title{Isoperimetric inequalities for eigenvalues of the Laplacian on cycles with fixed resistance metric} 
\author{Federico Men\'endez-Conde \thanks{Electronic address: \texttt{fmclara@uaeh.edu.mx}}}
\date{}
\affil{Centro de Investigaci\'on en Matem\'aticas, ICBI, UAEH}

\begin{document}

\maketitle

\begin{abstract}
For cycles with non-negative weights on its edges, we define its {\it global resistance} as the sum of the 
distances given by the effective resistance metric between adjacent vertices. We prove the following result: 
for the Laplace operator on the 3-cycle with global resistance equal to a given constant,
the maximal value of the smallest positive eigenvalue and the minimal value of the largest eigenvalue, are 
both attained if and only if all the weights are equal to each other.
\end{abstract}
\medskip 

\medskip 

\section{Introduction}

The study of isoperimetric inequalities for eigenvalues of the Laplacian and more general elliptic operators has a long history
(see, for example \cite{IIMPS, II, EPfEEO}). Back in the 1920's, Georg Faber and Edgar Krahn proved independently of each other,
a conjecture stated decades earlier by John W.S. Rayleigh. This conjecture established that, among all the planar domains with 
fixed area, the circle is the one with least first Dirichlet eigenvalue. That result, known as the Faber--Krahn inequality, is usually regarded 
as a starting point and a cornerstone of this theory. 
\medskip 

For eigenvalues of graph Laplacians, the study of isoperimetric inequalities is much more recent than the one for Laplacians on manifolds. 
In \cite{Friedman}, J. Friedman introduced the concept of the boundary of a graph and considered Faber--Krahn type problems for the eigenvalues 
of the Laplacian for regular trees. Different isoperimetric problems for eigenvalues in that context have been studied by various authors such as 
\cite{Leydold, Biyikoglu-Leylold_1,BLS, Katsuda-Urakawa}. In those works, the aim is to find the graph for which the Laplacian eigenvalues
are extremal when some chosen volume of the graph (say, its number of vertices or edges) is fixed. Minimizers of the first eigenvalue 
have been found for several classes of graphs. Just to mention one of these results, among other similar results, 
in \cite{Biyikoglu-Leylold_1} it is shown that, within the class of trees with fixed number of interior and boundary vertices, 
the minimal Dirichlet eigenvalue is attained for a comet (a star-shaped graph with a long tail), for which the boundary is the set of 
vertices with degree 1. That minimizer is unique up to isomorphism.
\medskip 
 
The present work is somewhat different to those referred in the previous paragraph in a couple of aspects. First, we consider graphs without  
boundary. Second, here the graphs considered are held fixed, and what varies are the weights on its edges. More precisely, we look at the 
eigenvalues of the usual Laplace operator on cycles, and the geometric quantity to be fixed 
(the {\it global resistance} presented in definition \ref{global_resistance}) is one that depends on the weights. This quantity is defined in 
terms of a metric of the graph known as the effective resistance metric (or just resistance metric) that was defined in \cite{Doyle-Snell} and 
originally inspired by electrical network theory. The effective resistance metric plays a key role in Jun Kigami's construction of the Laplacian on 
self-similar fractals, and for the construction of a natural metric on those fractal sets (see \cite{AoF,DEF}). 
\medskip 

Our work is organized as follows. In section \ref{erm} we present the resistance metric and introduce notation. In section \ref{gr} we define
the global resistance and make some calculations in this regard.
The main result of this work concerns the 3-cycle and is given in theorem \ref{isoperimetric_ineq}. It establishes 
that once the global resistance is fixed,  
the smallest positive eigenvalue of the Laplacian is maximized and the largest eigenvalue 
is minimized when the weights on the edges are all equal. The proof of this theorem is contained in section \ref{ii}. 
In section \ref{fc} we comment briefly on the plausible generalization of this result to larger $n$-cycles.

\section{The Effective Resistance Metric}\label{erm}
In this section we introduce notation and recall some basic definitions and well known facts from the analysis of graph Laplacians.
\medskip 

Let $\Gamma$ be a finite, simple and undirected graph with no loops. Denote by $\{v_0,\dots, v_{n-1}\}$ the vertices of $\Gamma$.  
Given a pair of vertices $v_i$ and $v_j$ connected to each other, we will write $v_i\sim v_j$. For such a pair, 
$c_{i,j}> 0$ will denote the weight on the edge between them and, 
following a well known electrical network analogy (originally introduced in \cite{Doyle-Snell}), 
we will call the $c_{i,j}$ the {\it conductances} of the graph.

It is convenient to define $c_{i,j}=0$ whenever the vertices $v_i$ and $v_j$ are not connected by an edge. In particular, the no loops 
condition means that $c_{i,i}=0$. 
We denote by $\ell(\Gamma)$ the set of real-valued functions defined on the vertices, and for $f\in \ell(\Gamma)$ we consider the 
usual norm and inner product. That is,
$$
\|f\|^2=f(v_0)^2+\cdots+f(v_{n-1})^2,\qquad \pint f g = f(v_0)g(v_0)+\cdots+f(v_{n-1})g(v_{n-1}).
$$

The Laplacian on $\Gamma$ is the linear operator on $\ell(\Gamma)$ given by
\begin{align*}
 (\Delta f)(v_k)=\sum_{j=0}^{n-1} c_{j,k}(f(v_k)-f(v_j)).
\end{align*}
It is well known that $\Delta$ is a non-negative operator, so that in particular it is self-adjoint. The 
associated quadratic form $E(f,g)=\pint {\Delta f}{g}$ is known as the {\it energy} of the graph. 
The expression $E(f)=E(f,f)$ defines a norm on the $n-1$ dimensional subspace of $\ell(\Gamma)$ that is orthogonal 
to the space of constant functions. This {\it energy norm} can also be written in the form 
\begin{equation}\label{energy}
 E(f) = \sum_{i\leq j} c_{i,j}(f(v_i)-f(v_j))^2. 
\end{equation}
\medskip 

The resistance metric between vertices of a weighted graph is defined in terms of the energy, as follows.
\medskip 

\begin{definition}
 Let $v_i$ and $v_j$ be two vertices in a weighted graph $\Gamma$. 
 \begin{align}
  m&=\min\left\{E(f)\ |\ f(v_i)=1, f(v_j)=0\right\}.\label{armonico} \\
  \ &\ \\
   d_{\rm r}(v_i,v_j)&=\frac 1 m.\nonumber
 \end{align}
The expression $d_{\rm r}(\cdot,\cdot)$ is called the effective resistance metric.
 \end{definition}

We refer to \cite{AoF,DEF} for the proof that the effective resistance metric is indeed a metric. 

It is known that the minimum in \eqref{armonico} is attained by the so-called {\it harmonic extension}. That is,
the unique function such that $f(v_i)=1$, $f(v_j)=0$  and  $(\Delta f)(v_k)=0$ for all $k\geq 2$ 
(see e.g. Theorem 2.1.6 in \cite{AoF} for a proof and more details of this).
This provides us with a formulation of the effective resistance metric that is computationally convenient:

Let's consider the distance between the first two vertices, namely $d_{\rm r}(v_0,v_1)$. 
This implies no loss of generality, since we can re-label the vertices as needed. 
The function $f$ with mininal energy $E(f)$ is the corresponding harmonic extension of $f(v_0)=1$ and $f(v_1)=0$ , and it  satisfies
$(\Delta f)(v_k)=0$ for all $k\geq 2$.
Let $H$ be the matrix representation of $\Delta$ with trespect to the canonical basis $\{e_0,\dots,e_{n-1}\}$ 
(where $e_j(v_k)=\delta_{j,k}$). 
\medskip 

Consider the decomposition
$$
H=
\begin{pmatrix}
 M & J^{t} \cr
 J &  L
\end{pmatrix},
\qquad\qquad 
M=\begin{pmatrix}
 \sum_k c_{0,k} & -c_{0,1} \cr
 -c_{0,1} &  \sum_k c_{1,k}
\end{pmatrix}.
$$
Let $f$ be the corresponding harmonic extension and split it as
$$
f =
\begin{pmatrix}
 f_0 \cr f_1
\end{pmatrix}
\qquad
f_0 = 
\begin{pmatrix}
1 \cr 0       
\end{pmatrix}
\qquad
f_1 = 
\begin{pmatrix}
f(v_2) \cr \vdots \cr f(v_{n-1})       
\end{pmatrix}
$$
The condition $(\Delta f)(v_k)=0$ gives $Jf_0+Lf_1=0$ so that $f_1$ is determined by $f_1=-L^{-1}Jf_0$. This in turn, gives 
an explicit formula for the minimal value of the energy (and hence for the metric $d_{\rm r}(v_0,v_1)$) as 
\begin{align}
 E(f_0) &= \pint {(M-J^t L^{-1} J)f_0}{f_0}.\label{hext}\\  
 d_{\rm r}(v_0,v_1) &= 1/E(f_0).\nonumber
 \end{align}

We note that if we multiply all the conductances by a constant $\alpha$, then $M-J^tL^{-1}J$ (and therefore $E(f_0)$ as well) 
is multiplied by the same constant. The resistance metric is then inversely proportional to $\alpha$. 

\section{Global Resistance in Cycles}\label{gr}

In this section we introduce the {\it global resistance}. This will be the geometric quantity to be fixed in the
isoperimetric problems to be considered in section \ref{ii}. 

\begin{definition}\label{global_resistance}
 For a weighted graph $\Gamma$ define the global resistance $\rho(\Gamma)$ by
 $$
 \rho(\Gamma) = \sum_{v_i\sim v_j} d_{\rm r}(v_i,v_j) 
 $$
 where the sum is taken over all un-ordered pairs of adjacent vertices. 
\end{definition}
To avoid any possible confussion, we remark that even though the effective resistance metric is defined 
for any pair of vertices in the graph, in the definition of the global resistance we only consider adjacent pairs. 
\medskip 

For the 3-cycle, the operators in \eqref{hext} are given by 
$$
M = 
\begin{pmatrix}
 c_{0,1}+c_{0,2} & -c_{0,1} \cr
 -c_{0,1} & c_{0,1}+c_{1,2}
\end{pmatrix},
\qquad
J =
\begin{pmatrix}
 -c_{0,2} & -c_{1,2}
\end{pmatrix},
\qquad 
L=\Bigl(c_{0,2}+c_{1,2}\Bigr)
$$
From this, the metric can be easily calculated to be  
$$
d_{\rm r}(v_0,v_1) = \frac{c_{0,2}+c_{1,2}}{c_{0,1}c_{0,2}+c_{0,1}c_{1,2}+c_{0,2}c_{1,2}}
$$
By symmetry we can see that 
$$
d_{\rm r}(v_0,v_2) = \frac{c_{0,1}+c_{1,2}}{c_{0,1}c_{0,2}+c_{0,1}c_{1,2}+c_{0,2}c_{1,2}}
$$
and
$$
d_{\rm r}(v_1,v_2) = \frac{c_{0,1}+c_{0,2}}{c_{0,1}c_{0,2}+c_{0,1}c_{1,2}+c_{0,2}c_{1,2}}.
$$
\medskip 

\begin{center}
\begin{tikzpicture}[shorten >=1pt,->]
  \tikzstyle{vertex}=[circle,fill=black!25,minimum size=12pt,inner sep=2pt]
  \node[vertex] (v0) at (0,2) {$v_0$};
  \node[vertex] (v1) at (-1,0)   {$v_1$};
  \node[vertex] (v2) at (1,0)  {$v_2$};
  \draw (v0) -- (v1) -- (v2) -- (v0) --  cycle;
  \node(c01) at (-1,1) {$c_{0,1}$};
    \node(c02) at (1,1) {$c_{0,2}$};
  \node(c12) at (0,-0.3) {$c_{1,2}$};
\end{tikzpicture}
\end{center}
\medskip 

The global resistance of the 3-cycle then becomes
\begin{equation}\label{resistencia_global}
\rho(T)=\frac{2(c_{0,1}+c_{0,2}+c_{1,2})}{c_{0,1}c_{0,2}+c_{0,1}c_{1,2}+c_{0,2}c_{1,2}}
\end{equation}

Note that if we hold fixed two of the conductances, then $\rho(T)$ is decreasing 
with respect to the third conductance. This simple observation will be useful later. 

\section{Isoperimetric Inequalities for 3-Cycles}\label{ii}

In this section we will prove our main result, namely:

\begin{theorem}\label{isoperimetric_ineq}
 Let $\lambda_1\leq\lambda_2$ be the positive eigenvalues of the Laplace operator associated with a 
 weighted 3-cycle $T$. Then the following inequalities hold:
 $$
\lambda_1\rho(T)\leq 6\leq \lambda_2\rho(T).
 $$
 In each of both sides, the equality occurs if and only if the weights are all equal to each other.
\end{theorem}

The proof will follow from three lemmas, which might be interesting on their own right, as they show that 
the eigenvalues depend on the conductances in a nice and simple way, as soon as the global resistance is fixed. 
The first of these lemmas, deals with the case where two of the conductances are equal to some constant $b$. In this case,
there is the notable and rather surprising property that, even though the eigenvalues vary,
the corresponding eigenvectors stay the same. Indeed, for the case where all the conductances are equal to $1$, we have that $\lambda=3$ is an 
eigenvalue with $\{(1,-1,0),(1,1,-2)\}$ an orthogonal basis of its eigenspace; those vectors will also be eigenvectors for any choice of $b$.  
More precisely:

\begin{lemma}\label{dos_iguales}
 Let $T$ be a 3-cycle with two conductances equal to some $b>0$ and global resistance $\rho(T)=2.$ 
 
 \begin{compactenum}
 \item 
  If $b\geq 1$ and $b=c_{0,2}=c_{1,2}$, then $\{(1,1,1),(1,-1,0),(1,1,-2)\}$ are eigenvectors
 corresponding to the (ordered) eigenvalues $0=\lambda_0<\lambda_1\leq\lambda_2$.
 \item 
 If $b\leq 1$ and $b=c_{0,2}=c_{1,2}$, then the order of the above eigenvectors for $\lambda_1$ and $\lambda_2$ is reversed, 
 with $(1,1,-2)$ being eigenvalue for $\lambda_1$, and $(1,-1,0)$ eigenvalue for $\lambda_2$.
 \item 
 Let $T$ be a 3-cycle for either of the above situations and any value of $b$. The eigenvalues of $T$ satisfy 
 $$
 \lambda_1\leq 3\leq\lambda_2
 $$
 with equality  if and only if all the conductances are equal to $1$.
\end{compactenum}
  \end{lemma}

\begin{proof}
 Substitution of $b=c_{0,2}=c_{1,2}$ in \eqref{resistencia_global} yields
 $$
 c_{0,1}=\frac{b(2-b)}{2b-1}.
 $$
 so that the positivity condition for the conductance restricts the possible values of $b$ to the interval $(1/2,2]$. 

\begin{center}
\begin{tikzpicture}[shorten >=1pt,->]
  \tikzstyle{vertex}=[circle,fill=black!25,minimum size=12pt,inner sep=2pt]
  \node[vertex] (v0) at (0,2) {$v_0$};
  \node[vertex] (v1) at (-1,0)   {$v_1$};
  \node[vertex] (v2) at (1,0)  {$v_2$};
  \draw (v0) -- (v1) -- (v2) -- (v0) --  cycle;
  \node (c01) at (-1.3,1) {$\frac{b(2-b)}{2b-1}$};
    \node (c02) at (0.8,1) {$b$};
  \node (c12) at (0,-0.3) {$b$};
\end{tikzpicture}
\end{center}
%\caption{Cycle with two equal conductances, as considered in lemma \ref{dos_iguales}}.

 The Laplacian matrix is in this case
 $$
H = \begin{pmatrix}
 \displaystyle{\frac{b(b+1)}{2b-1}} & \displaystyle{-\frac{b(2-b)}{2b-1} }& -b \cr
 \ &\ &\ \cr
  \displaystyle{-\frac{b(2-b)}{2b-1}} & \displaystyle{\frac{b(b+1)}{2b-1}} & -b \cr
   \ &\ &\ \cr
  -b & -b & 2b
   \end{pmatrix}
 $$
Direct computation shows that 
$$
H\begin{pmatrix}
  \phantom{-}1\cr -1\cr \phantom{-}0
 \end{pmatrix}
=\frac{3b}{2b-1}
\begin{pmatrix}
  \phantom{-}1\cr -1\cr \phantom{-}0
 \end{pmatrix}
\qquad 
{\rm and}\qquad
H\begin{pmatrix}
  \phantom{-}1\cr \phantom{-}1\cr -2
 \end{pmatrix}
=3b
\begin{pmatrix}
  \phantom{-}1\cr \phantom{-}1\cr -2
 \end{pmatrix}.
$$
 Since $3b/(2b-1)\leq 3b$ if and only if $b\geq 1$, then $\lambda_1=3b/(2b-1)$ for $b\leq 1$ and $\lambda_1=3b$ for $b\geq 1$.  
 So, $\lambda_1$ is increasing as a function of $b\in (1/2,1]$ and decreasing for $b\geq 1$ and we can see that its maximum value is 
 attained at $b=1$. The assertion for $\lambda_2$ follows in just the same way. 
 \end{proof}

 \begin{figure}
 \begin{center}
  \includegraphics[scale=0.59]{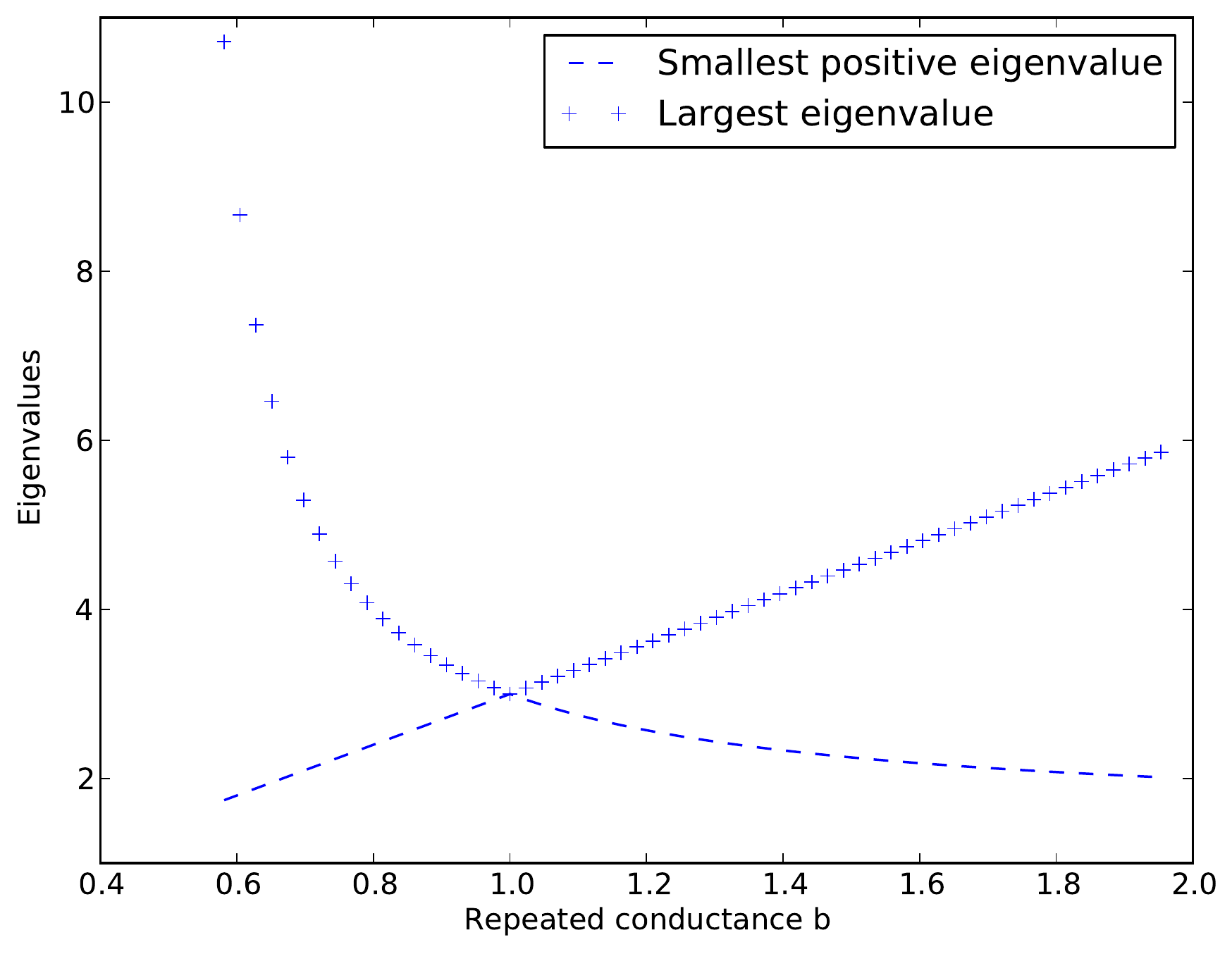}
  \caption{The two positive eigenvalues of the Laplacian of a 3-cycle conductances $\{b,b,(b(2-b))/(2b-1)\}$.}
 \label{dibujo_dos_iguales}
 \end{center}
 \end{figure}
 
 The behaviour of the eigenvalues in the previous lemma is shown in figure \ref{dibujo_dos_iguales}.
 \medskip 
  
 The next lemma establishes that when the middle value of the conductances is held fixed, we can make the maximal eigenvalue $\lambda_2$ 
 smaller by approaching one of the other two conductances to that middle value.
 \medskip 
 
 \begin{lemma}\label{lambda2}
 Let $T$ be a weighted 3-cycle with vertices $\{v_0,v_1,v_2\}$, global resistance $\rho(T)=2$, and denote by $c_{i,j}$ 
 the conductance on the edge joining $v_i$ with $v_j$. 
  
  \begin{compactenum}
    \item \label{c-grande}
      Supposse that 
      $$
      c_{0,1}<1\leq c_{0,2}\leq c_{1,2}=r.
      $$
      and fix $c_{0,2}=b$ for some $b\geq 1$. 
      Let $\lambda_2(r)$ be the largest eigenvalue of the Laplacian associated with $T$.
      Then $\lambda_2(r)$ is an increasing function of $r\geq b$.
   \item \label{c-chico} 
     Supposse that 
     $$
     r=c_{1,2}\leq c_{0,2}\leq 1<c_{0,1}.
     $$ 
     and fix $c_{0,2}=b\leq 1.$ Then, $\lambda_2(r)$ is a decreasing function of $r\leq b$.
  \end{compactenum}
 \end{lemma}

 \begin{proof} 
  First, we give a characterization for the eigenvector associated with $\lambda_2$, for conductances
  satisfying $c_{0,1}\leq c_{0,2}\leq c_{1,2}$ and being otherwise arbitrary.
  For a function $u\in\ell(T)$ given by $u(v_j)=x_j$, the energy \eqref{energy} takes the form
  \begin{equation}\label{3cycle_energy}
   E((x_0,x_1,x_2))=c_{0,1}(x_0-x_1)^2+c_{0,2}(x_0-x_2)^2+c_{1,2}(x_1-x_2)^2.  
   \end{equation}
  
  The Schwarz quotient 
  \begin{equation}
      S(u)=\frac{E(u)}{\|u\|^2}\label{schwarz}
  \end{equation}
  reaches its maximum value exactly for the functions $u$ such that $\Delta u=\lambda_2 u$.
  Note that for any permutation of the triad $(x_0,x_1,x_2)$ the denominator in \eqref{schwarz} does not change, while the energy will 
  be maximal when the largest distance $|x_i-x_j|$ is next to the largest conductances and the smallest distance
  $|x_i-x_j|$ is next to the smallest conductances. Therefore a necessary condition for $\Delta u=\lambda_2 u$ is that
  $$
  |x_0-x_1|\leq|x_0-x_2|\leq|x_1-x_2|
  $$
  From the orthogonality condition $x_0+x_1+x_2=0$, we can see that the term with largest absolute value 
  must have different sign than the other two, so that it must be involved in the two largest distances. 
  So, we obtain the necessary condition $|x_0|\leq|x_1|\leq|x_2|$  
  for $(x_0,x_1,x_2)$ being an eigenvector with eigenvalue $\lambda_2$. 
  
  Without loss of generality, we can always chose the 
  eigenvector to be such that
  \begin{equation}\label{caracterizar_e2}
   x_2<0<x_0\leq x_1<|x_2|=x_0+x_1\qquad {\rm and}\qquad x_0^2+x_1^2+x_2^2=1.
  \end{equation}
 
  For $b>1$ a given constant, the condition $\rho(T)=2$ determines $c_{0,1}$ as a function of the variable $r$. 
  
  Precisely,
    $$
      c_{0,1} = \frac{r(b-1)-b}{1-b-r}. 
    $$
\medskip     

  \begin{center}
    \begin{tikzpicture}[shorten >=1pt,->]
      \tikzstyle{vertex}=[circle,fill=black!25,minimum size=12pt,inner sep=2pt]
	\node[vertex] (v0) at (0,2.5) {$v_0$};
	\node[vertex] (v1) at (-1.25,0)   {$v_1$};
	\node[vertex] (v2) at (1.5,0)  {$v_2$};
	\draw (v0) -- (v1) -- (v2) -- (v0) --  cycle;
	\node (c01) at (-1.9,1.5) {$c_{0,1}=\frac{r(b-1)-b}{1-b-r}$};
	\node (c02) at (1.5,1.5) {$c_{0,2}=r$};
	\node (c12) at (0,-0.5) {$c_{1,2}=b$};
    \end{tikzpicture}
  \end{center}
\medskip

  Now, we suppose that $u(v_j)=x_j$ is an eigenvector of $\Delta_{r_0}$  for some fixed $r_0>a$, with eigenvalue $\lambda_2(r_0)$. 
  Assume also that $(x_0,x_1,x_2)$ satisfies the conditions \eqref{caracterizar_e2}.
  
  We want to show that 
   \begin{equation}\label{lambda2creciente}
    \lambda_2(r_0)\leq\lambda_2(r_1). 
   \end{equation}
  whenever or $1\leq b\leq r_1\leq r_0$.
  
  Note that
  \begin{align*}
    \lambda_2(r_0)& = E_{r_0}((x_0,x_1,x_2)) \\ 
    & = \frac{r_0(b-1)-b}{1-b-r_0}(x_0-x_1)^2+b(x_0-x_2)^2+r_0(x_1-x_2)^2\\
    &\ \\
    \lambda_2(r_1)& \geq E_{r_1}((x_0,x_1,x_2)) \\
    & = \frac{r_1(b-1)-b}{1-b-r_1}(x_0-x_1)^2+b(x_0-x_2)^2+r_1(x_1-x_2)^2. 
  \end{align*}
  Hence, in order to obtain \eqref{lambda2creciente} it suffices to verify that
  $$
  \frac{r_0(b-1)-b}{1-b-r_0}(x_0-x_1)^2+r_0(x_1-x_2)^2\leq\frac{r_1(b-1)-b}{1-b-r_1}(x_0-x_1)^2+r_1(x_1-x_2)^2. 
  $$  
  To see this, we will show that 
  \begin{equation}\label{lagama}
  \gamma(r)= \frac{r(b-1)-b}{1-b-r}A+rB
  \end{equation}
  is increasing for $r\geq b$, provided that $0<A\leq B$. Since by conditions \eqref{caracterizar_e2} 
  we have $(x_0-x_1)^2\leq (x_1-x_2)^2$ this would give the result. 
  
  Now,
   \begin{align*}
   \gamma'(r)& = \frac{(b-1+r)^2 B-\left((b-1)^2+b\right)A}{(1-b-r)^2} \\
   &\ \\
   & \geq \frac{(b-1+r)^2 -\left((b-1)^2+b\right)}{(1-b-r)^2}A.
  \end{align*}
  But
  \begin{align*}
  (b-1+r)^2 & 	\geq (2b-1)^2 \\ 
	    & = b^2 + (b-1)^2 +2b(b-1) \\
	    & > (b-1)^2+b.
  \end{align*}
  We conclude that $\gamma'>0$ and \eqref{lambda2creciente} follows. 
  \medskip
 
  To prove part \eqref{c-chico}, first we observe that the change of order of the conductances means that the roles of $x_2$ and
  $x_0$ are interchanged, so that now we take the eigenvector for $\lambda_2$ to be such that
  \begin{equation}\label{caracterizar_e2_bis}
   x_0<0<x_2\leq x_1<|x_0|=x_1+x_2\qquad {\rm and}\qquad x_0^2+x_1^2+x_2^2=1.
  \end{equation}
    
    Take an eigenvector $(x_0,x_1,x_0)$ with eigenvalue $\lambda_2$ for some $r_0\leq b\leq 1$, satisfying \eqref{caracterizar_e2_bis}.
    For an arbitrary $r_2<r_0$ we want to show that $\lambda(r_2)>\lambda(r_0)$. Proceeding just like in part \eqref{c-grande} it follows
    that it is enough prove the inequality 
    $$
    \frac{r_0(b-1)-b}{1-b-r_0}(x_0-x_1)^2+r_0(x_1-x_2)^2\leq\frac{r_2(b-1)-b}{1-b-r_2}(x_0-x_1)^2+r_2(x_1-x_2)^2. 
    $$ 
    Noting that $x_0=-x_1-x_2$ we can see that 
    \begin{align*}
     (x_0-x_1)^2&=(2x_1-x_2)^2\\
     &=((x_1-x_2)+x_1)^2\\
     &\geq 4(x_1-x_2)^2.
     \end{align*}
      In view of this, to get the inequality $\lambda(r_2)>\lambda(r_0)$ it is enough to verify that the function $\gamma(r)$ in 
      \eqref{lagama} is decreasing whenever $A\geq 4B$. This is equivalent to  the inequality
      $$
      (b-1+r)^2 B < 4((b-1)^2+b)A.      
      $$
  It is clear that it is enough to check the extremal case $B=4A$ and $r=a$
    $$
    (2b-1)^2<4((b-1)^2+b)
    $$
  which can be easily verified.     
  \end{proof}
 \medskip 
 
 \begin{figure}
 \begin{center}
  \includegraphics[scale=0.59]{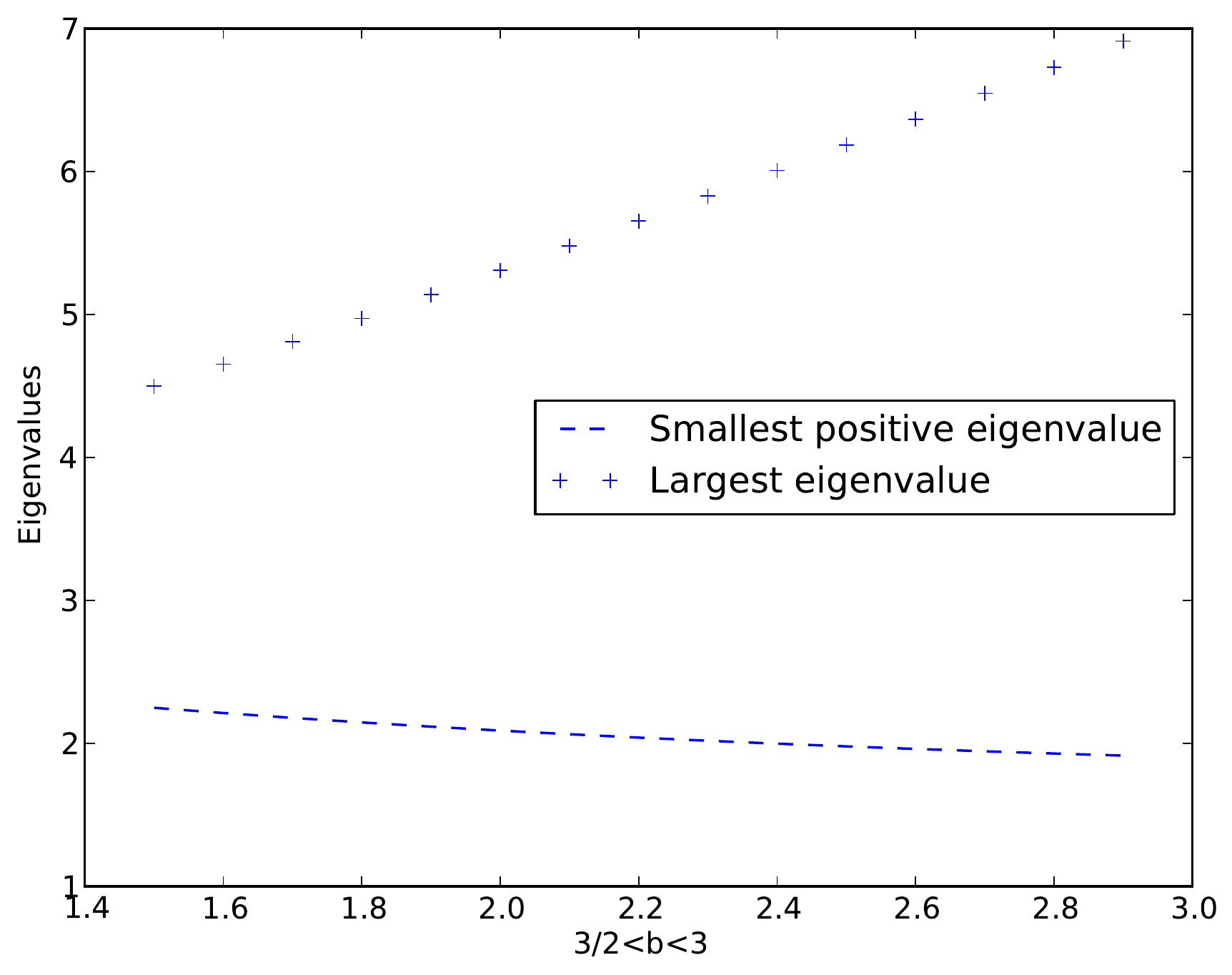}
  \caption{The two positive eigenvalues of the Laplacian of a 3-cycle with conductances equal to $\{(3-r)/(r+1),3/2,r\}$.}
 \label{bgrande}
 \end{center}
 \end{figure}
 \medskip 
 
 The analogous result for $\lambda_1$ is as follows. The ideas in the proof of lemma \ref{lambda1} are very similar to those of
 lemma \ref{lambda2}, so we will not go into as much detail.
 \medskip 
 
 \begin{lemma}\label{lambda1}
 Let $T$ be a weighted 3-cycle with vertices $\{v_0,v_1,v_2\}$, global resistance $\rho(T)=2$, and denote by $c_{i,j}$ 
 the conductance on the edge joining $v_i$ with $v_j$. 
  
  \begin{compactenum}
    \item \label{c-grande1}
      Supposse that 
      $$
      c_{0,1}<1\leq c_{0,2}\leq c_{1,2}=r.
      $$
      and fix $c_{0,1}=b$ for some $b\geq 1$. 
      Let $\lambda_1(r)$ be the largest eigenvalue of the Laplace operator associated with $T$.
      Then $\lambda_1(r)$ is a decreasing function for $r\geq b$.
   \item \label{c-chico1} 
     Supposse that 
     $$
     r=c_{1,2}\leq c_{0,2}\leq 1<c_{0,1}.
     $$ 
     and fix $c_{0,1}=b\leq 1.$ Then $\lambda_1(r)$ is increasing for $0<r\leq b$.
  \end{compactenum}
 \end{lemma}

 \begin{proof}
  If $u=(y_0,y_1,y_2)$ is an eigenvector for $T$ corresponding to $\lambda_1$, then it minimizes de Schwarz quotient \eqref{schwarz} 
  among all the elements in $\ell(T)$ that are orthogonal to the constant functions. 
  Opposite to the maximum, the minimal energy \eqref{3cycle_energy} will be attained when the largest of the $y_j$ is next to the smallest 
  conductances, and the smallest $y_j$ is next to the largest conductances. Thus, for part \ref{c-grande1}, we can consider that   
  \begin{equation}\label{caracterizar_e1}
  y_0< 0 < y_2\leq y_1 < |y_0| = y_1+y_2\qquad {\rm and}\qquad  y_0^2+y_1^2+y_2^2=1.
  \end{equation}
 
  Proceeding as we did in the proof of lemma \ref{lambda2}, proving that $\lambda_1(r)$ is decreasing for $r\geq b\geq 1$ reduces 
  to verify that 
  \begin{equation*}
   \gamma(r)= \frac{r(b-1)-b}{1-b-r}(y_0-y_1)^2+r(y_1-y_2)^2
  \end{equation*}
  is decreasing. But, since $(y_0-y_1)^2\geq 4(y_1-y_2)^2$, we already have shown (in the last part of the proof of lemma \ref{lambda2}) 
  that such is the case. 
  \medskip 
  
  For part \ref{c-chico1} we can take the eigenvector $(y_0,y_1,y_2)$ to satisfy the conditions \eqref{caracterizar_e1}.  
  The fact that $\lambda_1(r)$ is increasing for $0<r\leq a\leq 1$ follows from 
  $$
  \gamma(r)= \frac{r(b-1)-b}{1-b-r}(y_0-y_1)^2+r(y_1-y_2)^2
  $$
  being increasing, which is satisfied since $(y_0-y_1)^2\leq (y_1-y_2)^2$.
 \end{proof}

 See figures \ref{bgrande} and \ref{bchico} for ilustrations of particular cases of lemmas \ref{lambda2} and \ref{lambda1}. 
 \medskip 
 
 \begin{proof}[Proof of theorem \ref{isoperimetric_ineq}]
 From the observation at the end of section \ref{erm} we have that $\rho (\alpha T)$ is inversely proportional to $\alpha$. Since the 
 eigenvalues $\lambda_j$ are directely proportional to $\alpha$, we have that the products $\rho (\alpha T)\lambda_j$ do not change with $\alpha$.
 In other words, those products depend exclusively on the proportions between the conductances, not on their absolute values. So,
 without loss of generality, we can consider only the case when $\rho(T)=2$.
 
 Given $T_0$ the 3-cycle with conductances equal to 1, we want to conclude that whenever $\rho(T)=2$ the inequalities 
 \begin{equation}\label{result}
  \lambda_1(T)\leq\lambda_1(T_0)=\lambda_2(T_0)\leq\lambda_2(T)
 \end{equation}
    hold. 
    
 It is straightforward to verify that if we leave two of the 
 conductances fixed, then $\rho(T)$ decreases as the third conductance increases. Therefore, in order to satisfy $\rho(T)=2$ it is 
 necessary that, if $a<b<c$ are the conductances, we have either of the two situations 
 \begin{align}
  a&\leq 1\leq b\leq c\label{cond_big}\\
  a&\leq b\leq 1\leq c.\label{cond_small} 
 \end{align}
  Supposse that for $T$ we have the case \eqref{cond_big}. 
  %As, in lemma \ref{dos_iguales}, we have already shown \eqref{result}  
  %for the case $b=c$, we can assume that $b<c$. 
  Define $\tilde T$ to be the 3-cycle with conductances $a,b,b$. 
  Then, from part \ref{c-grande} of lemma \ref{lambda2} and part \ref{c-grande1} of lemma \ref{lambda1} it follows that 
  $$
  \lambda_1(T)\leq\lambda_1(\tilde T)\leq\lambda_2(\tilde T)\leq\lambda_2(T).
  $$
  And by lemma \ref{dos_iguales} we know that 
  $$
  \lambda_1(\tilde T)\leq\lambda_1(T_0)=\lambda_2(T_0)\leq\lambda_2(\tilde T),
  $$
  so that we obtain \eqref{result}. If $T$ corresponds to the other case \eqref{cond_small} the result 
  follows in analogous way.
 \end{proof}
 \medskip 
 
  \begin{figure}
  \begin{center}
  \includegraphics[scale=0.59]{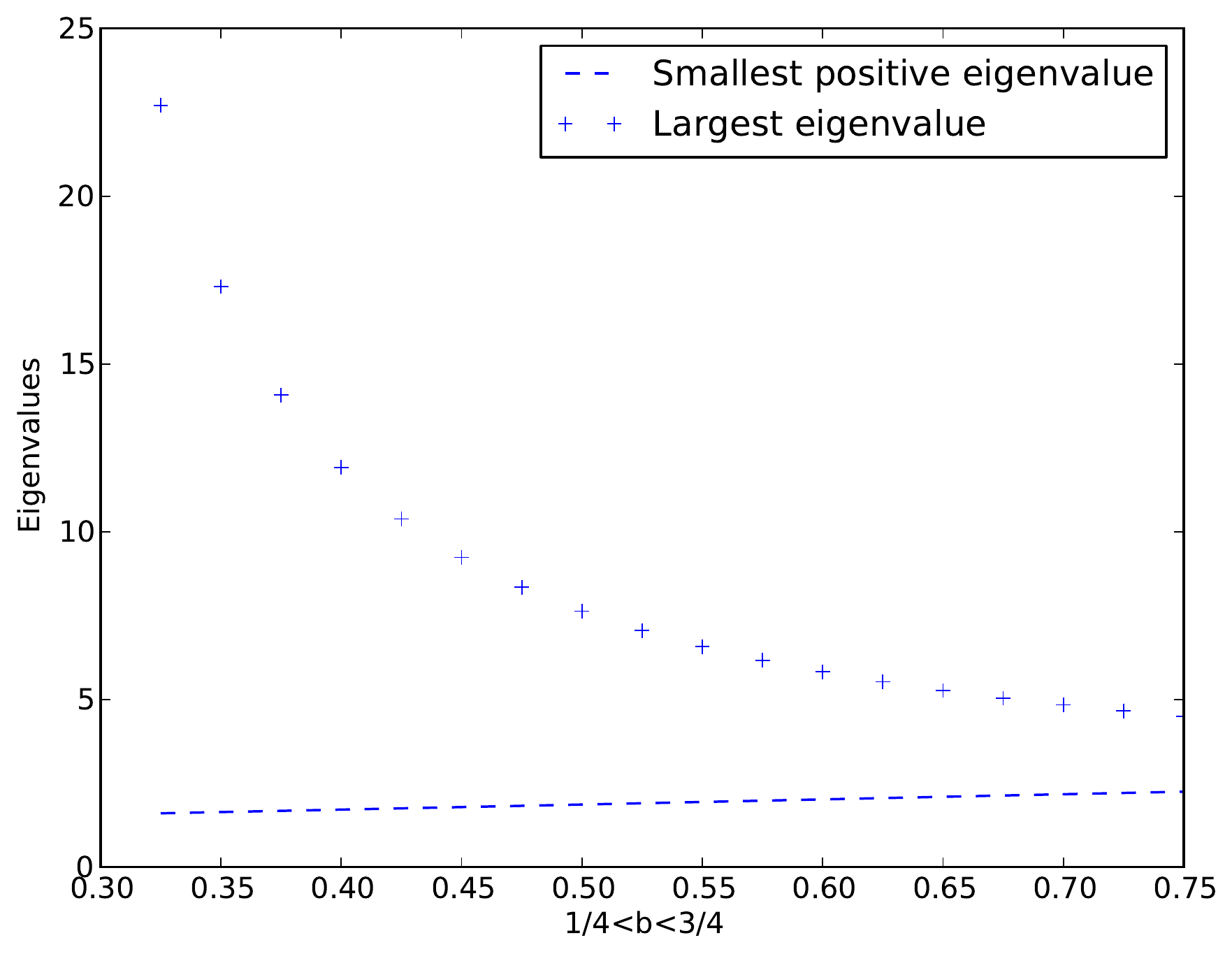}
  \caption{The two positive eigenvalues of the Laplacian of a 3-cycle with conductances equal to $\{3/4,r,(r+3)/(4r-1)\}$.}
 \label{bchico}
 \end{center}
 \end{figure}
 \medskip

\section{Further considerations}\label{fc}

\begin{figure}
 \begin{center}
  \includegraphics[scale=0.59]{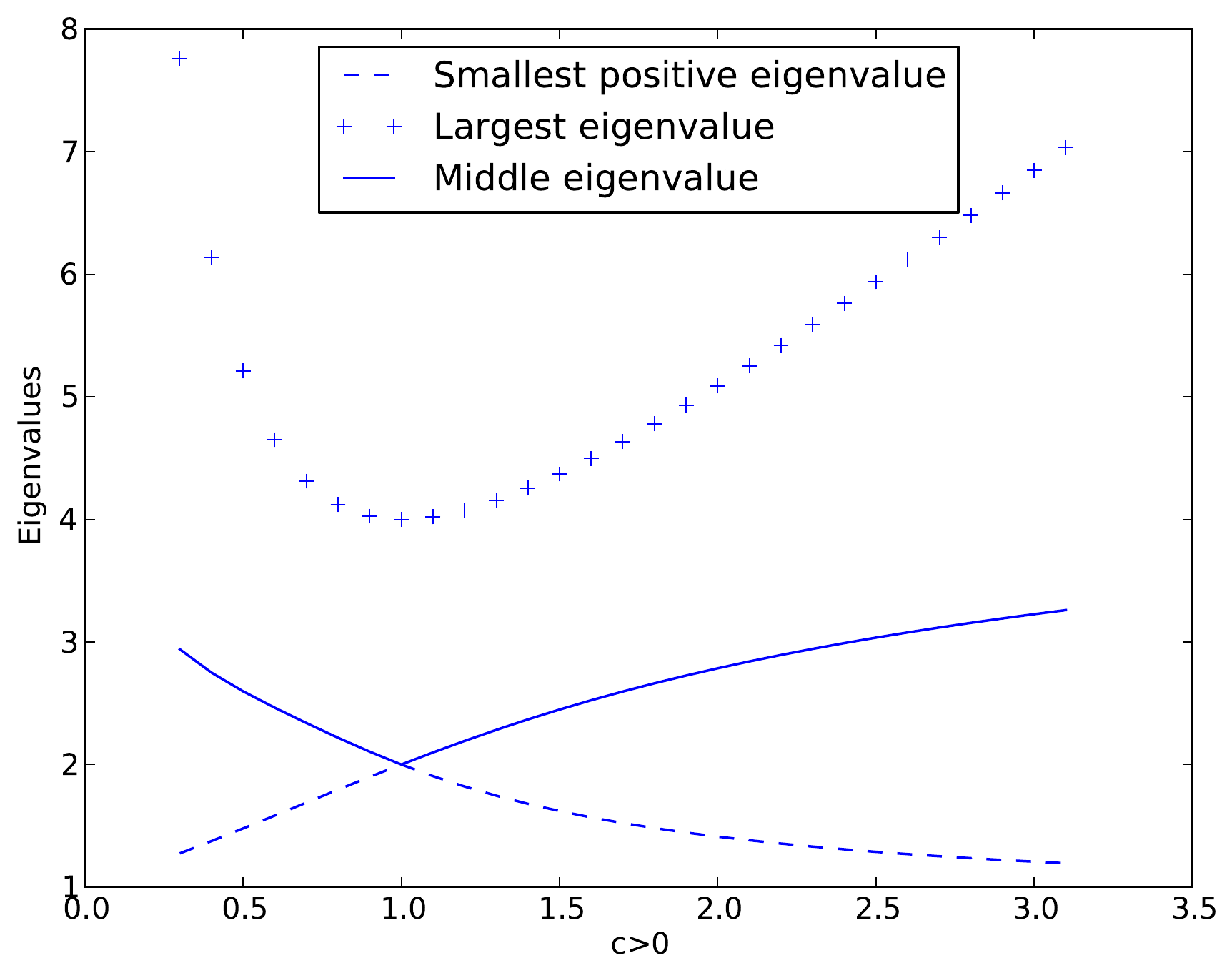}
  \caption{The three positive eigenvalues of the Laplacian of a 4-cycle with conductances $1$, $c$, $1/c$ and $(1+c^2-c)/(1+c^2+c)$.}
 \label{n4_especial1}
 \end{center}
 \end{figure}

It is reasonable to question whether the results presented are also true for the general cycle with $n$ vertices. More precisely one would expect 
that, for every $n$, if $T_0$ is the Laplacian for the $n$-cycle with constant conductances then
\begin{equation}\label{conjetura}
\lambda_1(T)\leq\lambda_1(T_0)\leq\lambda_{n-1}(T_0)\leq\lambda_{n-1}(T_0).  
\end{equation}
for every Laplacian $T$ on the cycle with $\rho(T)=\rho(T_0)$. So far, we have not found a counterexample for that, and the numerical evidence also points 
in that direction. However, it does not seem plausible that the methods used for the case $n=3$ can be adapted in a simple way to 
establish the general case. Most likely, a different approach might be needed to obtain a general proof.
\medskip 

Hereby, we show graphical evidence in two particular situations for the $4$-cycle. 
It is straightforward to calculate the global resistance of the $n$ cycle with conductances $c_{i,j}$ to be equal to
$$
\rho(T)=\frac{2(c_{0,1}c_{1,2}+c_{0,1}c_{2,3}+c_{0,1}c_{0,3}+c_{1,2}c_{2,3}+c_{1,2}c_{0,3}+c_{2,3}c_{0,3})}
{c_{0,1}c_{1,2}c_{2,3}+c_{0,1}c_{1,2}c_{0,3}+c_{1,2}c_{2,3}c_{0,3}+c_{1,2}c_{2,3}c_{0,3}}
$$

This gives $\rho(T_0)=3$. The condition $\rho(T)=3$ determines each conductance as a function of the other three, in particular
$$
c_{0,1}=\frac{3c_{1,2}c_{2,3}c_{0,3}-2(c_{1,2}c_{2,3}+c_{1,2}c_{0,3}+c_{2,3}c_{0,3})}
{2(c_{1,2}+c_{2,3}+c_{0,3})-3(c_{1,2}c_{2,3}+c_{1,2}c_{0,3}+c_{2,3}c_{0,3}))}
$$

For the particular example when shown in the next picture, the behaviour of the eigenvalues is shown in the plot of figure \ref{n4_especial1}.
\medskip 

\begin{center}
    \begin{tikzpicture}[shorten >=1pt,->]
      \tikzstyle{vertex}=[circle,fill=black!25,minimum size=12pt,inner sep=2pt]
	\node[vertex] (v0) at (2,2) {$v_0$};
	\node[vertex] (v1) at (-2,2) {$v_1$};
	\node[vertex] (v2) at (-2,0) {$v_2$};
	\node[vertex] (v3) at (2,0) {$v_3$};
	\draw (v0) -- (v1) -- (v2) -- (v3) --(v0) --  cycle;
	\node (c01) at (-0.05,2.5) {$c_{0,1}=\frac{1+c^2-c}{1+c^2+c}$};
	\node (c12) at (-3.1,1) {$c_{1,2}=1/c$};
	\node (c23) at (0,-0.5) {$c_{2,3}=c$};
	\node (c03) at (3,1) {$c_{2,3}=1$};
    \end{tikzpicture}
\end{center}
\medskip 

The second case considered is:
\medskip 

\begin{center}
    \begin{tikzpicture}[shorten >=1pt,->]
      \tikzstyle{vertex}=[circle,fill=black!25,minimum size=12pt,inner sep=2pt]
	\node[vertex] (v0) at (2,2) {$v_0$};
	\node[vertex] (v1) at (-2,2) {$v_1$};
	\node[vertex] (v2) at (-2,0) {$v_2$};
	\node[vertex] (v3) at (2,0) {$v_3$};
	\draw (v0) -- (v1) -- (v2) -- (v3) --(v0) --  cycle;
	\node (c01) at (-0.05,2.5) {$c_{0,1}$};
	\node (c12) at (-3.1,1) {$c_{1,2}=1/c$};
	\node (c23) at (0,-0.5) {$c_{2,3}=c$};
	\node (c03) at (3.7,1) {$c_{2,3}=(c+1)/2$};
    \end{tikzpicture}
\end{center}
\medskip

 \begin{figure}
 \begin{center}
  \includegraphics[scale=0.59]{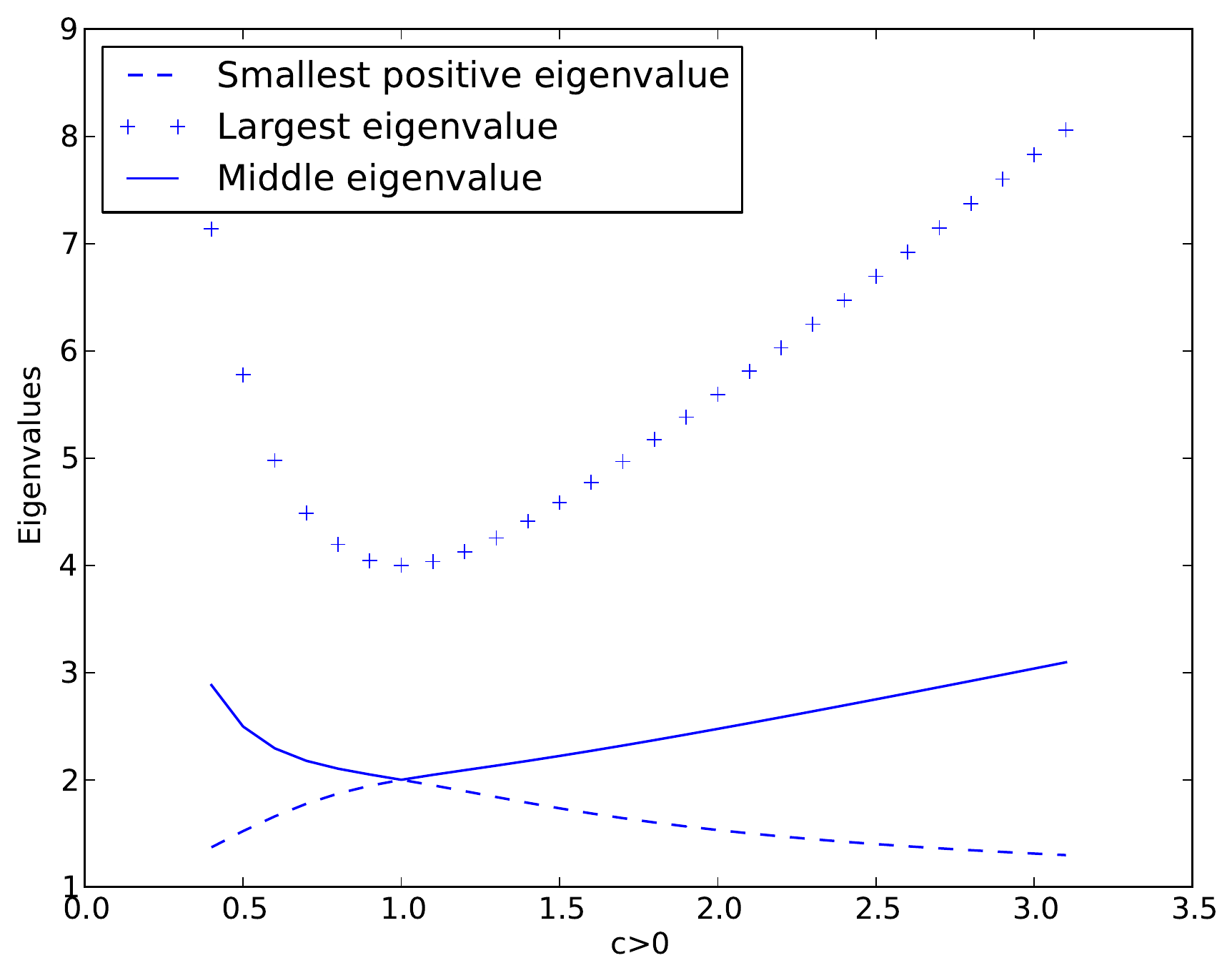}
  \caption{The three positive eigenvalues of the Laplacian of a 4-cycle with conductances $1$, $c$, $1/c$ and $(1+c^2-c)/(1+c^2+c)$.}
 \label{n4_especial2}
 \end{center}
 \end{figure}

The plot of the eigenvalues for this case is shown in figure \ref{n4_especial2}.
\medskip 

As we see, these plots suggest that the eigenvalues have a similar behaviour to the one of the case $n=3$. Namely, that not only \eqref{conjetura}
 might be true, but also that we could expect nice monotonic behaviours similar to the ones described in lemmas \ref{dos_iguales}, \ref{lambda2}
 and \ref{lambda1}.

\newpage
 
\bibliographystyle{plain}

\end{document}